\documentclass[reqno,11pt]{amsart}
\usepackage[foot]{amsaddr}
\usepackage{bbold}
\usepackage{charter}
\usepackage[margin=1in]{geometry}
\usepackage[colorlinks=true,linktoc=all,linkcolor=blue,citecolor=blue]{hyperref}
\usepackage{tikz}
\tikzstyle{every node}=[font=\tiny]

\allowdisplaybreaks

\newcommand{\C}{\mathbb{C}}
\newcommand{\D}{\mathbb{D}}
\renewcommand{\H}{\mathbb{H}}
\newcommand{\N}{\mathbb{N}}
\newcommand{\R}{\mathbb{R}}
\newcommand{\Aut}{\mathrm{Aut}(\D)}

\theoremstyle{plain}
\newtheorem{theorem}{Theorem}[section]
\newtheorem*{theorem*}{Theorem}
\newtheorem{proposition}[theorem]{Proposition}
\newtheorem{lemma}[theorem]{Lemma}
\newtheorem{corollary}[theorem]{Corollary}
\theoremstyle{definition}

\newtheorem*{definition*}{Definition}
\newtheorem{remark}[theorem]{Remark}

\newtheorem{question}{Question}
\numberwithin{equation}{section}
\everymath{\displaystyle}
\allowdisplaybreaks

\title[Spectrum of a Composition Operator]{Spectrum of a Composition Operator with Automorphic Symbol}
\author{Robert F.~Allen\textsuperscript{1}, Thong M.~Le\textsuperscript{2}, and Matthew A.~Pons\textsuperscript{3}}
\address{\textsuperscript{1}Department of Mathematics, University of Wisconsin-La Crosse}
\address{\textsuperscript{2}Department of Computer Science, University of California}
\address{\textsuperscript{3}Department of Mathematics, North Central College}
\email{rallen@uwlax.edu, thmle@ucdavis@edu, mapons@noctrl.edu}

\date{}

\begin{document}

\begin{abstract}
We give a complete characterization of the spectrum of composition operators, induced by an automorphism of the open unit disk, acting on a family of Banach spaces of analytic functions that includes the Bloch space and BMOA.  We show that for parabolic and hyperbolic automorphisms, the spectrum is the unit circle.  For the case of elliptic automorphisms, the spectrum is either the unit circle or a finite cyclic subgroup of the unit circle.  
\end{abstract}

\maketitle

\section{Introduction}
For an analytic self-map $\varphi$ of the open unit disk $\D$ and a Banach space $X$ of functions analytic on $\D$, we define the composition operator with symbol $\varphi$, denoted $C_\varphi$, by the rule $C_\varphi f = f\circ\varphi$ for all $f \in X$.  The study of composition operators began formally with Nordgren's paper \cite{Nordgren:1968}, where he explored properties of composition operators acting on the Hardy Hilbert space $H^2$. Since then, the study has proved to be an active area of research, most likely due to the fact that the study of such operators lies at the intersection of complex function theory and operator theory.

The spectrum of $C_\varphi$ has been studied on many classical spaces of analytic functions, such as the Hardy spaces, Bergman spaces, weighted Hardy and Bergman spaces, Besov spaces, and the Dirichlet space.  The interested reader is directed to \cite{CowenMacCluer:1995} for general references. 

The motivation of this paper was to determine the spectrum of a composition operator, induced by a disk automorphism, acting on the Bloch space.  The Bloch space is the largest space of analytic functions on $\D$ that is M\"obius invariant.  This is one reason the Bloch space is a welcoming environment to study composition operators.  The techniques developed apply to a larger class of spaces that includes the Bloch space. 

The purpose of this paper is to determine the spectrum of $C_\varphi$ acting on a family of Banach spaces, where $\varphi$ is a disk automorphism.  The spectrum will depend on the fixed point classification of the automorphisms of $\D$.  This is a standard approach to the study of composition operators induced by automorphisms.  We show the spectrum of $C_\varphi$, acting on a particular family of Banach spaces, induced by a disk automorphism, must be a subset of the unit circle $\partial \D$, and in some instances is the entire unit circle.  Finally, we compare these results to particular examples of classical spaces.  

\section{Preliminaries}\label{Section:Preliminaries}
\subsection{Automorphisms}\label{Subsection:Automorphisms}
The automorphisms of the open unit disk $\D = \{z \in \C: |z| < 1\}$ are precisely the analytic bijections on $\D$ which have the form
$$\varphi(z) = \lambda\frac{a-z}{1-\overline{a}z}$$ where $\lambda$ is a unimodular constant and $a$ is a point in $\D$.  These automorphisms form a group under composition denoted by $\Aut$.  Every element of $\Aut$ has two fixed points (counting multiplicity), and thus can be classified by the location of the fixed points: 
\begin{description}
	\item[elliptic] one fixed point in $\D$ and one in the complement of $\overline{\D}$;
	\item[parabolic] one fixed point on the unit circle $\partial \D$ (of multiplicity 2); 
	\item[hyperbolic] two distinct fixed points on $\partial \D$.
\end{description}

Two disk automorphisms $\varphi$ and $\psi$ are conformally equivalent if there exists a disk automorphism $\tau$ for which $\psi = \tau\circ\varphi\circ\tau^{-1}$.  Many properties of automorphisms are preserved under conformal equivalence.  The main advantage of conformal equivalence is in the placement of the fixed points.  Every elliptic disk automorphism is conformally equivalent to one whose fixed point in $\D$ is the origin.  

\begin{lemma}\label{Lemma:elliptic_confequiv} Let $\varphi$ be an elliptic disk automorphism with fixed point $a$ in $\D$.  Then $\varphi$ is conformally equivalent to $\psi(z) = \lambda z$ where $\lambda = \varphi'(a)$.
\end{lemma}

\begin{proof}
	Let $\tau_a$ be the involution automorphism which interchanges 0 and $a$, that is $$\tau_a(z) = \frac{a-z}{1-\overline{a}z}.$$  Define $\psi = \tau_a\circ\varphi\circ\tau_a^{-1}$ on $\D$.  Since $a$ is a fixed point of $\varphi$, $\psi$ fixes the origin, and thus is a rotation.  So there exists an unimodular constant $\lambda$ such that $\psi(z) = \lambda z$.  To complete the proof, we will show $\lambda = \varphi'(a)$.  Observe $\psi'(z) = \lambda$ for all $z \in \D$.  In particular
	$$\lambda = \psi'(0) = \tau_a'(\varphi(\tau_a(0)))\varphi'(\tau_a(0))\tau_a'(0) = \varphi'(a)\tau_a'(a)\tau_a'(0) = \varphi'(a).$$  Thus $\varphi$ is conformally equivalent to the rotation $\psi(z) = \varphi'(a) z$.
\end{proof}

Every parabolic disk automorphism is conformally equivalent to one whose fixed point (of multiplicity 2) is 1. The following Lemma is found as Exercise 2.3.5c of \cite{CowenMacCluer:1995}, and a complete proof can be found in \cite{Pons:2007}. 

\begin{lemma}\cite[Lemma 4.1.2]{Pons:2007}\label{Lemma:parabolic_confequiv} Let $\varphi$ be a parabolic disk automorphism.  Then $\varphi$ is conformally equivalent to either $\psi_1(z) = \displaystyle\frac{(1+i)z-1}{z+i-1}$ or $\psi_2(z) = \displaystyle\frac{(1-i)z-1}{z-i-1}$.
\end{lemma} 

Every hyperbolic disk automorphism is conformally equivalent to one whose fixed points in $\partial \D$ are $\pm 1$.

\begin{lemma}\cite[Theorem 6]{Nordgren:1968}\label{Lemma:hyperbolic_confequiv} Let $\varphi$ be a hyperbolic disk automorphism.  Then $\varphi$ is conformally equivalent to $\psi(z) = \displaystyle\frac{z+r}{1+rz}$ for some $r \in (0,1)$.
\end{lemma}

\subsection{The Space of Bounded Analytic Functions}\label{Subsection:Bloch}
The set of analytic functions on $\D$ is denoted by $H(\D)$.  The space of bounded analytic functions on $\D$, denoted $H^\infty = H^\infty(\D)$, is a Banach space under the norm $$\|f\|_\infty = \sup_{z \in \D}\,|f(z)|.$$  The bounded analytic functions on $\D$ is a rich space containing many interesting types of functions, such as polynomials and Blaschke products.  In addition, the disk algebra $\mathcal{A}(\D)$, the set of analytic functions on $\D$ continuous to $\partial\D$, is a closed subspace of $H^\infty$.

The following two families of functions will be used in the next section.  To prove these functions are in $H^\infty$, we take a geometric approach using conformal mappings of the plane.  To this effect, let $\H_\ell$ and $\H_r$ denote the open left and right half planes respectively, i.e. $\H_\ell = \{\mathrm{Re}\, z < 0\}$ and $\H_r = \{\mathrm{Re}\, z > 0\}$.

\begin{lemma}\label{Lemma:exp_bloch} For $s\ge 0$, the function $f_s(z) = \exp\left(\displaystyle\frac{s(z+1)}{z-1}\right)$ is in $H^\infty$.\end{lemma}

\begin{proof}
	If $s = 0$, then $f_s(z) = 1$. So, $f_s(z)$ is in $H^\infty$. Now suppose $s > 0$. The function $f_s$ is comprised of the functions
	\begin{enumerate}
		\item $z \mapsto \frac{z+1}{z-1}$; mapping $\D$ onto $\H_\ell$,
		\item $z \mapsto sz$; mapping $\H_\ell$ onto $\H_\ell$,
		\item $z \mapsto e^z$; mapping $\H_\ell$ onto $\D\setminus\{0\}$.
	\end{enumerate}
	
	\begin{center}
		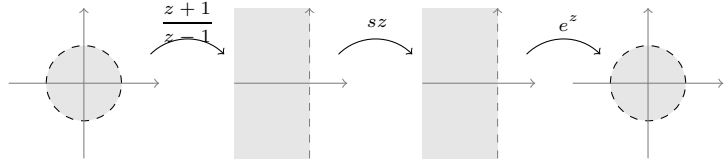
\begin{figure}[h]
			\begin{tikzpicture}
			\draw[fill=gray!20!white, dashed] (0,0) circle (.5);  
			\draw[->, gray] (-1,0) -- (1,0);
			\draw[->, gray] (0,-1) -- (0,1);
			
			\draw (.75,.25) node (d1) {};
			\draw (2,.25) node (d2) {};
			\draw[->] (d1) to [bend left=45] (d2);
			\node [] at (1.375,.8) {$\frac{z+1}{z-1}$};
			
			\fill [gray!20!white] (2,-1) rectangle (3,1);
			\draw[->, gray] (2,0) -- (3.5,0);
			\draw[->, gray, dashed] (3,-1) -- (3,1);
			
			\draw (3.25,.25) node (d3) {};
			\draw (4.5,.25) node (d4) {};
			\draw[->] (d3) to [bend left=45] (d4);
			\node [] at (3.9,.8) {$sz$};
			
			\fill [gray!20!white] (4.5,-1) rectangle (5.5,1);
			\draw[->, gray] (4.5,0) -- (6,0);
			\draw[->, gray, dashed] (5.5,-1) -- (5.5,1);
			
			\draw (5.75,.25) node (d5) {};
			\draw (7,.25) node (d6) {};
			\draw[->] (d5) to [bend left=45] (d6);
			\node [] at (6.45,.8) {$e^z$};
			
			\draw[fill=gray!20!white, dashed] (7.5,0) circle (.5);  
			\draw[->, gray] (6.5,0) -- (8.5,0);
			\draw[->, gray] (7.5,-1) -- (7.5,1);
			\end{tikzpicture}
			\caption{Map $f_s(z) = \exp\left(\displaystyle\frac{s(z+1)}{z-1}\right)$ for $s > 0$.}\label{fig: map1}
		\end{figure}
	\end{center}
	So $f_s$ maps $\D$ into $\D$, as depicted in Figure \ref{fig: map1}, and thus $f_s(z)$ is an element of $H^\infty$.
\end{proof}

\begin{lemma}\label{Lemma:log_bloch} For real value $t$, the function $f_t(z) = \left(\displaystyle\frac{1+z}{1-z}\right)^{it}$ is in $H^\infty$.
\end{lemma}

\begin{proof}
	For $t = 0$, $f_t$ is identically 1, and thus is in $H^\infty$.  Now suppose $t > 0$.  We will rewrite the function $f_t$ as $$f_t(z) = \exp\left(it\log\frac{1+z}{1-z}\right),$$ where $\log$ is the principle branch of the logarithm. Then $f_t$ is comprised of the functions
	\begin{enumerate}
		\item $z \mapsto \frac{1+z}{1-z}$; mapping $\D$ onto $\H_r$,
		\item $z \mapsto \log z$; mapping $H_r$ onto the horizontal strip $S_h = \{0 < \mathrm{Im}\,z < 2\pi \}$,
		\item $z \mapsto itz$; mapping $S_h$ onto the vertical strip $S_v = \{-2\pi  < \mathrm{Re}\, z < 0\}$,
		\item $z \mapsto e^z$; mapping $S_v$ into $A(e^{-2\pi},1) = \{e^{-2\pi} < |z| < 1\}$.
	\end{enumerate}
	
	\begin{center}
		\begin{figure}[h]
			\begin{tikzpicture}
			\draw[fill=gray!20!white, dashed] (0,0) circle (.5);  
			\draw[->, gray] (-1,0) -- (1,0);
			\draw[->, gray] (0,-1) -- (0,1);
			
			\draw (.7,.5) node (d1) {};
			\draw (1.8,.5) node (d2) {};
			\draw[->] (d1) to [bend left=45] (d2);
			\node [] at (1.275,1) {$\frac{1+z}{1-z}$};
			
			\fill [gray!20!white] (2,-1) rectangle (3,1);
			\draw[->, gray] (1.5,0) -- (3,0);
			\draw[->, gray, dashed] (2,-1) -- (2,1);
			
			\draw (3,.5) node (d3) {};
			\draw (4.1,.5) node (d4) {};
			\draw[->] (d3) to [bend left=45] (d4);
			\node [] at (3.55,1) {$\log z$};
			
			\fill [gray!20!white] (4,0) rectangle (6,.5);
			\draw[gray, dashed] (4,.5) -- (6,.5);
			\draw[->, gray, dashed] (4,0) -- (6,0);
			\draw[->, gray] (5,-1) -- (5,1);
			
			\fill [gray!20!white] (4,0) rectangle (6,.5);
			\draw[gray, dashed] (4,.5) -- (6,.5);
			\draw[->, gray, dashed] (4,0) -- (6,0);
			\draw[->, gray] (5,-1) -- (5,1);
			
			\draw (5.8,.5) node (d3) {};
			\draw (6.9,.5) node (d4) {};
			\draw[->] (d3) to [bend left=45] (d4);
			\node [] at (6.4,1) {$itz$};
			
			\fill [gray!20!white] (7,-1) rectangle (7.5,1);
			\draw[gray, dashed] (7,-1) -- (7,1);
			\draw[->, gray] (6.5,0) -- (8,0);
			\draw[->, gray, dashed] (7.5,-1) -- (7.5,1);
			
			\draw (7.8,.5) node (d3) {};
			\draw (8.9,.5) node (d4) {};
			\draw[->] (d3) to [bend left=45] (d4);
			\node [] at (8.4,1) {$e^z$};
			
			\draw[fill=gray!20!white, dashed] (9.5,0) circle (.5); 
			\draw[fill=white, dashed] (9.5,0) circle (.25); 
			\draw[->, gray] (8.5,0) -- (10.5,0);
			\draw[->, gray] (9.5,-1) -- (9.5,1);
			\end{tikzpicture}
			\caption{Map $f_t(z) = \exp\left(it\log\displaystyle\frac{1+z}{1-z}\right)$ for $t > 0$.}\label{fig: map2}
		\end{figure}
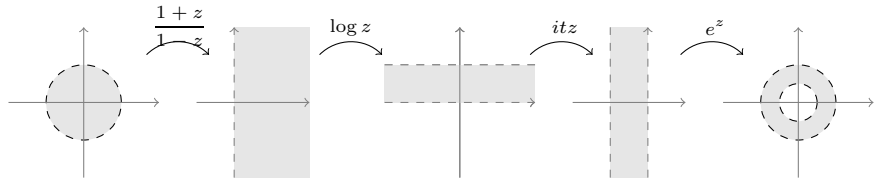
	\end{center}
	So $f_t$ maps $\D$ into $A(e^{-2\pi},1) \subseteq \D$, as depicted in Figure \ref{fig: map2}.  In the case of $t < 0$, the vertical strip $S_v$ becomes $\{0 < \mathrm{Re}\, z < 2\pi \}$.  The map $z \mapsto e^z$ takes $S_v$ into $A(1,e^{2\pi}) \subseteq e^{2\pi}\D$, as depicted in Figure \ref{fig: map3}.
	\begin{center}
		\begin{figure}[h]
			\begin{tikzpicture}
			\draw[fill=gray!20!white, dashed] (0,0) circle (.5);  
			\draw[->, gray] (-1,0) -- (1,0);
			\draw[->, gray] (0,-1) -- (0,1);
			
			\draw (.7,.5) node (d1) {};
			\draw (1.8,.5) node (d2) {};
			\draw[->] (d1) to [bend left=45] (d2);
			\node [] at (1.275,1) {$\frac{1+z}{1-z}$};
			
			\fill [gray!20!white] (2,-1) rectangle (3,1);
			\draw[->, gray] (1.5,0) -- (3,0);
			\draw[->, gray, dashed] (2,-1) -- (2,1);
			
			\draw (3,.5) node (d3) {};
			\draw (4.1,.5) node (d4) {};
			\draw[->] (d3) to [bend left=45] (d4);
			\node [] at (3.55,1) {$\log z$};
			
			\fill [gray!20!white] (4,0) rectangle (6,.5);
			\draw[gray, dashed] (4,.5) -- (6,.5);
			\draw[->, gray, dashed] (4,0) -- (6,0);
			\draw[->, gray] (5,-1) -- (5,1);
			
			\draw (5.8,.5) node (d3) {};
			\draw (6.9,.5) node (d4) {};
			\draw[->] (d3) to [bend left=45] (d4);
			\node [] at (6.4,1) {$itz$};
			
			\fill [gray!20!white] (7,-1) rectangle (7.5,1);
			\draw[gray, dashed] (7.5,-1) -- (7.5,1);	
			\draw[->, gray, dashed] (7,-1) -- (7,1);
			\draw[->, gray] (6.5,0) -- (8,0);
			
			\draw (7.8,.5) node (d3) {};
			\draw (8.9,.5) node (d4) {};
			\draw[->] (d3) to [bend left=45] (d4);
			\node [] at (8.4,1) {$e^z$};
			
			\draw[fill=gray!20!white, dashed] (9.5,0) circle (.75); 
			\draw[fill=white, dashed] (9.5,0) circle (.5); 
			\draw[->, gray] (8.5,0) -- (10.5,0);
			\draw[->, gray] (9.5,-1) -- (9.5,1);
			\end{tikzpicture}
			\caption{Map $f_t(z) = \exp\left(it\log\displaystyle\frac{1+z}{1-z}\right)$ for $t < 0$.}\label{fig: map3}
		\end{figure}
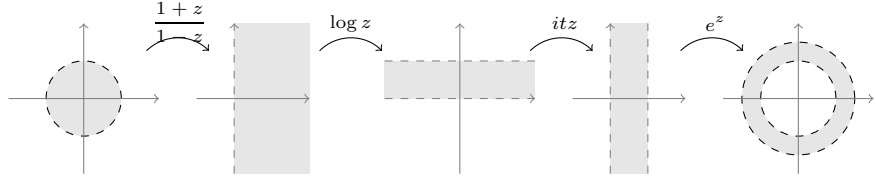
	\end{center}
	In either case, $f_t(z)$ is an element of $H^\infty$ since $\|f_t\|_\infty < e^{2\pi}$ for all $t \in \R$.
\end{proof}

These functions above, together with the monomials, play such a pivitol role in Section \ref{Section:Main} that we denote the union of these functions by $\mathcal{F}$, i.e. $$\mathcal{F} = \left\{f_s : s \geq 0\right\} \cup \left\{f_t : t \in \R \right\}\cup\left\{z^k : k \in \N\right\}.$$

\subsection{Spectrum of $C_\varphi$}\label{Subsection:Spectrum}
In this section we collect useful results regarding the spectrum of operators on Banach spaces.  For a bounded linear operator $T$ on a Banach space $X$, the spectrum of $T$ is given by $$\sigma(T) = \{\lambda \in \C : T-\lambda I \text{ is not invertible}\}$$ where $I$ denotes the identity operator on $X$.  The spectrum is a nonempty, closed subset of $\C$.  The spectral radius of $T$ is given by $$\rho(T) = \sup\,\{|\lambda| : \lambda \in \sigma(T)\}.$$  Due to the fact that the spectrum is closed, we have the spectrum of $T$ is contained in the closed disk centered at the origin of radius $\rho(T)$.  

Determining the spectrum of a particular composition operator can be difficult depending on the symbol of the operator and the space on which it is acting.  However, the difficulties can be avoided if the operator is similar to a ``simpler" operator.  Linear operators $S$ and $T$ (not necessarily bounded) on a Banach space $X$ are similar if there exists a bounded linear operator $U$ on $X$, having bounded inverse, such that $T = USU^{-1}$.  If $S$ and $T$ are both bounded operators, then similarity preserves the spectrum.

\begin{theorem}\label{Theorem:similar_spectrum} Let $S$ and $T$ be bounded operators on a Banach space $X$.  If $S$ and $T$ are similar, then $\sigma(S) = \sigma(T)$.
\end{theorem}

\begin{proof}
	Suppose $S$ and $T$ are similar operators on $X$. By definition, there exists an invertible, bounded operator $U$ such that $T = USU^{-1}$. Let $\lambda\in\C$ and observe that, 
	$$\begin{aligned}
	T-\lambda I &= USU^{-1}-\lambda I\\
	&= USU^{-1} - \lambda UU^{-1}\\
	&= USU^{-1} - U(\lambda I)U^{-1}\\
	&= U(S - \lambda I)U^{-1}.
	\end{aligned}$$  Thus, we have that $S-\lambda I$ is not invertible if and only if $T-\lambda I$ is not invertible.  Therefore $\sigma(S) = \sigma(T)$.
\end{proof}

\section{Main Results}\label{Section:Main}

In this section, we determine the spectrum of $C_\varphi$ for $\varphi$ a disk automorphism acting on a particular family of Banach spaces of analytic functions.  The spaces we consider will be denoted by $\mathcal{X}$ and have the following properties:
\begin{enumerate}
	\item[(i)] $\mathcal{X}$ contains $\mathcal{F}$,
	\item[(ii)] for all $\varphi \in \Aut$, $C_\varphi$ is bounded on $\mathcal{X}$ and $\rho(C_\varphi)=1$.
\end{enumerate} 
The set of automorphisms of $\D$, as seen previously, is a very nice subset of the analytic self-maps of $\D$.  By property (ii), every composition operator induced by a disk automorphism is bounded on $\mathcal{X}$.  In fact, every such composition operator is invertible.  This result, that we prove below, can be viewed as a consequence of Theorem 1.6 of \cite{CowenMacCluer:1995}.

\begin{proposition}\label{Theorem:inverse_automorphism} Let $\varphi$ be a disk automorphism and $C_{\varphi}$ the induced composition operator on $\mathcal{X}$. Then $C_{\varphi}$ is invertible with inverse $C^{-1}_{\varphi} = C_{\varphi^{-1}}$.
\end{proposition}

\begin{proof}
	Since $\varphi \in \text{Aut}(\D)$, $\varphi$ is invertible, and $\varphi^{-1}$ is an automorphism. The composition operator $C_{\varphi^{-1}}$ is bounded by property (ii) and 
	$$C_{\varphi} \left( C_{\varphi^{-1}}(f) \right) = C_{\varphi}( f \circ \varphi^{-1}) = f \circ \varphi^{-1} \circ \varphi = f$$
	$$C_{\varphi^{-1}} \left( C_{\varphi}(f) \right) = C_{\varphi^{-1}}( f \circ \varphi) = f \circ \varphi \circ \varphi^{-1} = f \ .$$
	Therefore, $C_{\varphi}$ is invertible with $C_{\varphi}^{-1} = C_{\varphi^{-1}}$. 
\end{proof}

Since the spectral radius of $C_\varphi$ on $\mathcal{X}$ is 1 for $\varphi \in \Aut$, we see that the search for the spectrum can be restricted to subsets of $\overline{\D}$.  However, our search can be refined further to subsets of the unit circle.

\begin{theorem}\label{Theorem:automorphism_spectrum_unitcircle} Let $\varphi$ be a a disk automorphism and $C_{\varphi}$ the induced composition operator on $\mathcal{X}$.  Then $\sigma(C_\varphi) \subseteq \partial\D$.
\end{theorem}

\begin{proof}
	By property (ii) of $\mathcal{X}$, we have $\rho(C_{\varphi}) = 1$. So, $\sigma(C_{\varphi}) \subseteq \overline{\D}$. Since, by Proposition \ref{Theorem:inverse_automorphism}, $C_{\varphi}$ is invertible with the inverse $C^{-1}_{\varphi} = C_{\varphi^{-1}}$, then $0\notin \sigma(C_{\varphi})$. So, the function $f(z) = z^{-1}$ is analytic in some neighborhood of $\sigma(C_{\varphi})$. By the Spectral Mapping Theorem (see Theorem 5.14 of \cite{MacCluer:2009}), we have $\sigma(f\circ C_{\varphi}) = f(\sigma(C_{\varphi}))$, and so,
	$$\sigma(C_{\varphi^{-1}}) = \sigma(C^{-1}_{\varphi}) = \sigma(C_{\varphi})^{-1} = \{\lambda^{-1}: \lambda \in \sigma(C_{\varphi})\}\ .$$
	Since $\varphi^{-1} \in \Aut$, $\sigma(C_{\varphi^{-1}}) \subseteq \overline{\D}.$  Thus for $\lambda \in \sigma(C_\varphi)$, both $\lambda$ and $\lambda^{-1}$ are in $\overline{\D}$.  This implies $\lambda \in \partial\D$.  So $\sigma(C_\varphi) \subseteq \partial\D$, as desired. 
\end{proof}

Since the disk automorphisms are classified into three categories, according to fixed points, we will treat each type of automorphism separately.  However, the strategy to determine $\sigma(C_\varphi)$ is the same.  For a disk automorphism $\varphi$, we have shown $\varphi$ to be conformally equivalent to a particularly ``nice" disk automorphism: in the elliptic case a disk automorphism that fixes 0, in the parabolic case a disk automorphism that fixes 1, and in the hyperbolic case a disk automorphism that fixes $\pm 1$.  In the next result, we show that conformally equivalent automorphisms induce similar composition operators on $\mathcal{X}$.  This result is not unique to the space $\mathcal{X}$, but is true for any space for which automorphisms induce bounded composition operators (see pg. 250 of \cite{CowenMacCluer:1995}).

\begin{proposition}\label{Theorem:confequiv_similar} Let $\varphi$ and $\psi$ be conformally equivalent disk automorphisms.  Then the induced composition operators $C_\varphi$ and $C_\psi$ on $\mathcal{X}$ are similar.
\end{proposition}

\begin{proof}
	Suppose $\varphi$ and $\psi$ are conformally equivalent disk automorphisms.  Then there exists a disk automorphism $\tau$ such that $\psi = \tau\circ\varphi\circ\tau^{-1}$.  For $f \in \mathcal{X}$, observe
	$$C_\psi f = f\circ(\tau\circ\varphi\circ\tau^{-1}) = ((f\circ\tau)\circ\varphi)\circ\tau^{-1} = (C_{\tau^{-1}} C_\varphi C_{\tau})f.$$  Since $C_{\tau^{-1}}$ is bounded and invertible on $\mathcal{X}$ with $C_{\tau^{-1}}^{-1} = C_{\tau}$, then $C_\psi = C_{\tau^{-1}} C_\varphi C_{\tau^{-1}}^{-1}$.  Therefore $C_\varphi$ and $C_\psi$ are similar.
\end{proof}

With Proposiition \ref{Theorem:confequiv_similar} and Lemmas \ref{Lemma:elliptic_confequiv}, \ref{Lemma:parabolic_confequiv}, and \ref{Lemma:hyperbolic_confequiv}, it suffices to determine the spectrum of the  ``nice" disk automorphisms, since similarity of bounded operators preserves the spectrum.

\begin{theorem}\label{Theorem:elliptic_spectrum} Let $\varphi$ be an elliptic disk automorphism with fixed point $a$ in $\D$.  Then the spectrum of $C_\varphi$ acting on $\mathcal{X}$ is the closure of the positive powers of $\varphi'(a)$.  Moreover, this closure is a finite subgroup of the unit circle if $\varphi'(a)^n = 1$ for some natural number $n$, and is the unit circle otherwise.
\end{theorem}

\begin{proof} By Lemma \ref{Lemma:elliptic_confequiv}, $\varphi$ is conformally equivalent to $\psi(z) = \lambda z$ where $\lambda = \varphi'(a)$.  By Proposition \ref{Theorem:confequiv_similar}, it suffices to show that $\sigma(C_\psi)$ is the closure of the positive powers of $\lambda$.  Let $G = \langle\lambda\rangle = \{\lambda^k : k \in \N\}$, which is a subset of $\partial\D$ since $|\lambda| = 1$.  For each $k \in \N$, the function $f_k(z) = z^k$ is in $\mathcal{X}$ by property (i), and we have $(C_\psi f_k)(z) = \lambda^kf_k(z)$.  Thus $\lambda^k$ is an eigenvalue of $C_\psi$ corresponding to the eigenfunction $f_k$.  So $G \subseteq \sigma(C_\psi)$, and since the spectrum is closed, we have $\overline{G} \subseteq \sigma(C_\psi) = \sigma(C_\varphi)$.  If the order of $\lambda$ is infinite, then $G$ is dense in $\partial \D$, and so $\overline{G} = \partial\D$.
	
Now suppose $\lambda$ has order $m < \infty$.  Then $G = \{\lambda^k : k = 1, \dots, m\}$.  So, $\overline{G} = G$.  We now wish to show $\sigma(C_\psi) \subseteq \overline{G}$.  Since $\sigma(C_\psi) \subseteq \partial\D$ by Theorem \ref{Theorem:automorphism_spectrum_unitcircle}, it suffices to show that if $\mu \in \partial\D\setminus\overline{G}$ then $\mu\notin \sigma(C_\psi)$.  Suppose $\mu \in \partial\D\setminus\overline{G}$.
	
Since $\mu \notin \overline{G}$, it clear that $\mu \notin G$ and $\mu^m \neq 1$.  In order to show $\mu \notin \sigma(C_{\psi})$, we will show that $C_{\psi}-\mu I$ is invertible by proving that for every $g\in \mathcal{X}$, there exists a unique $f\in \mathcal{X}$ such that $f\circ \psi - \mu f = g$.
	
Since the order of $\lambda$ is $m$, we have $\psi^{(m)}(z) = (\underbrace{\psi \circ \dots \circ \psi}_{m\mathrm{-times}})(z)= \lambda^m z = z.$  By repeated composition with $\psi$, we obtain the system of linear equations:
	$$\begin{aligned}
	f\circ\psi 	- \mu f 					&= g\\
	f\circ\psi^{(2)} 	- \mu (f\circ\psi) 			&= g\circ\psi\\
	&\;\,\vdots \\
	f 					- \mu (f\circ\psi^{(m-1)}) 	&= g\circ\psi^{(m-1)} \ .
	\end{aligned}$$ 
	This system of linear equations can be expressed as the matrix equation $A\vec{x}=\vec{b}$ where
	$$A = \left[ \begin{matrix} 
	-\mu & 1 & 0 & 0 & \cdots & 0\\
	0 	& -\mu & 1 & 0 &\cdots & 0\\
	\vdots  & 0 &\ddots & \ddots & &\vdots\\
	\vdots  & \vdots &\ddots & \ddots &\ddots &\vdots\\
	0 	& \vdots &  &\ddots &\ddots & 1\\
	1 & 0 & \cdots & \cdots & 0 & -\mu\\
	\end{matrix}\right],
	\vec{x} = \left[\begin{matrix}f\\f\circ\psi\\\vdots\\\vdots\\f\circ\psi^{(m-2)}\\f\circ\psi^{(m-1)}\end{matrix}\right], \text{ and } 
	\vec{b} = \left[\begin{matrix}g\\g\circ\psi\\\vdots\\\vdots\\g\circ\psi^{(m-2)}\\g\circ\psi^{(m-1)}\end{matrix}\right].$$
	The determinant of $A$ is $(-1)^m(\mu^m-1)$, which is not zero since $\mu \notin G$. Thus there is a unique solution for $\vec{x}$. It gives us the unique solution $f$, which is a finite linear combination of function in $\mathcal{X}$ of the form $g\circ \psi^{(j-1)}$ for $j = 1, \dots, m$, and thus $f$ is in $\mathcal{X}$. It follows that $C_{\psi}-\mu I$ is invertible. So, $\mu\notin \sigma(C_{\psi})$. Therefore, $\sigma(C_{\varphi}) = \sigma(C_\psi) \subseteq \overline{G}$.
\end{proof}

\begin{theorem}\label{Theorem:parabolic_spectrum} Let $\varphi$ be a parabolic disk automorphism.  Then the spectrum of $C_\varphi$ acting on $\mathcal{X}$ is the unit circle.
\end{theorem}

\begin{proof}  From Lemma \ref{Lemma:parabolic_confequiv}, $\varphi$ is conformally equivalent to either $\psi_1(z) = \displaystyle\frac{(1+i)z-1}{z+i-1}$ or $\psi_2(z) = \displaystyle\frac{(1-i)z-1}{z-i-1}$.  By Theorems \ref{Theorem:confequiv_similar} and \ref{Theorem:automorphism_spectrum_unitcircle} it suffices to show that $\partial\D$ is a subset of $\sigma(C_{\psi_1})$ and $\sigma(C_{\psi_2})$.
	
	First suppose $\varphi$ is conformally equivalent to $\psi_1$.  Consider the function $$f_s(z) = \exp\left(\displaystyle\frac{s(z+1)}{z-1}\right)$$
	for $s\geq 0$.  By property (i), $f_s$ is in $\mathcal{X}$. Observe 
	$$\begin{aligned}
	(C_{\psi_1} f_s)(z) &= f_s(\psi_1(z)) = f_s\left( \frac{(1+i)z-1}{z+i-1} \right)\\
	&= \exp \left( \frac{s \left( \frac{(1+i)z-1}{z+i-1} +1 \right)}{\frac{(1+i)z-1}{z+i-1} - 1}\right) = \exp \left( \frac{s((1+i)z-1+z+i-1)}{(1+i)z-1-z-i+1}\right)\\
	&= \exp \left( \frac{s((2+i)z+i-2)}{i(z-1)}\right) = \exp \left( \frac{s((1-2i)z+1+2i)}{z-1}\right)\\
	&= \exp \left( \frac{s(z+1)}{z-1} - 2is \right)\\
	&= e^{i(-2s)}f_s(z).
	\end{aligned}$$
	So, $f_s$ is an eigenfunction of $C_{\psi_1}$ for $s\geq 0$. Then, $\partial\D = \{e^{i(-2s)}: s\geq 0\}$ is a subset of $\sigma(C_{\psi_1})$.  If $\varphi$ is conformally equivalent to $\psi_2$, then by a similar calculation, we have $$(C_{\psi_2} f_s)(z) = e^{2is}f_s(z),$$ and so $\partial\D = \{e^{2is}:s\geq 0\}$ is a subset of $\sigma(C_{\psi_2})$.  Therefore, $\sigma(C_\varphi) = \partial\D$, as desired.
\end{proof}

\begin{theorem}\label{Theorem:hyperbolic_spectrum} Let $\varphi$ be a hyperbolic disk automorphism.  Then the spectrum of $C_\varphi$ acting on $\mathcal{X}$ is the unit circle.
\end{theorem}

\begin{proof}
	From Lemma \ref{Lemma:hyperbolic_confequiv}, $\varphi$ is conformally equivalent to $\psi(z) = \displaystyle\frac{z+r}{1+rz}$ for some $r \in (0,1)$.  By Theorems \ref{Theorem:confequiv_similar} and \ref{Theorem:automorphism_spectrum_unitcircle} it suffices to show that $\partial\D \subseteq\sigma(C_\psi)$. Consider the function $$f_t(z) = \left(\frac{1+z}{1-z}\right)^{it}$$ for $t \in \R$.  By property (i), $f_t$ is in $\mathcal{X}$.  Observe
	$$\begin{aligned}
	(C_{\psi} f_t)(z) &= f_t(\psi(z)) = f_t\left(\frac{z+r}{1+rz}\right)\\
	&= \left(\frac{1+\frac{z+r}{1+rz}}{1-\frac{z+r}{1+rz}}\right)^{it} = \left(\frac{1+rz+z+r}{1+rz-z-r}\right)^{it}\\
	&= \left(\frac{(r+1)z+(r+1)}{(r-1)z-(r-1)}\right)^{it}\\
	&= \left(\frac{r+1}{r-1}\right)^{it}f_t(z).
	\end{aligned}$$  So, $f_t$ is an eigenfunction of $C_{\psi}$ for $t$ real.  Then $\partial\D = \left\{\left(\frac{r+1}{r-1}\right)^{it}: 0 < r < 1, t \in \R\right\}$ is a subset of $\sigma(C_{\psi}) = \sigma(C_\varphi)$.  Therefore $\sigma(C_\varphi) = \partial\D$, as desired.
\end{proof}

\section{Examples \& Comparisons}
In this section we first consider examples of spaces that satisfy the properties of $\mathcal{X}$.  For these spaces, our results characterize the spectrum of composition operators induced by disk automorphisms.  Lastly, we consider spaces that do not satisfy the properties of $\mathcal{X}$ but for which the spectrum of composition operators induced by automorphisms is known.  We will compare the spectra for those spaces with the characterization for $\mathcal{X}$.

\subsection{Examples}\label{Section:Examples}
First, we will discuss examples of spaces that satisfy the properties of $\mathcal{X}$.
\subsubsection{Bounded analytic functions}
The property (i) of $\mathcal{X}$ is satisfied by $H^\infty$ by Lemmas \ref{Lemma:exp_bloch} and \ref{Lemma:log_bloch}.  In fact, on $H^\infty$, any analytic self-map of $\D$ induces a bounded composition operator $C_\varphi$ such that $\|C_\varphi\| = 1$.  Equality is achieved since $H^\infty$ contains the constant function 1.  The spectral radius formula (see Theorem 5.15 of \cite{MacCluer:2009}) then implies that $\rho(C_\varphi) = 1$.  Thus, property (ii) is satisfied.  This $H^\infty$ belongs to the family of Banach spaces of analytic functions $\mathcal{X}$.

\subsubsection{Bloch space} The Bloch space on $\D$, denoted $\mathcal{B} = \mathcal{B}(\D)$, is the space of analytic functions on $\D$ such that $\beta_f = \displaystyle\sup_{z \in \D}\,(1-|z|^2)|f'(z)| < \infty$.  The quantity $\beta_f$ is a semi-norm, called the Bloch semi-norm.  The Bloch space is a Banach space under the norm $$\|f\|_\mathcal{B} = |f(0)| + \beta_f.$$  

It is well-known that $\mathcal{B}$ is a Banach space of analytic functions that contains $H^\infty$, and thus satisfies property (i) of $\mathcal{X}$.  In fact, every analytic self-map of $\D$ induces a bounded composition operator on $\mathcal{B}$ (see pg. 126 \cite{ArazyFisherPeetre:1985}).  Donaway, in his Ph.D. thesis, Corollary 3.9 of \cite{Donaway:1999}, proved the spectral radius of every composition operator induced by an analytic function on $\D$, and in particular the disk automorphisms, is 1.  So the Bloch space satisfies all the properties of $\mathcal{X}$.

\subsubsection{Analytic functions of bounded mean osciallation}\label{Section:BMOA}  The space of analytic functions on $\D$ with bounded mean oscillation on $\partial \D$, denoted $BMOA$, is defined to be the set of functions in $H(\D)$ such that $$\|f\|_* = \sup_{z \in \D}\,\|f\circ\tau_a - f(z)\|_{H^2} < \infty,$$
where $H^2$ is defined in Section \ref{Section:Hardy Spaces}.  The space $BMOA$ is a Banach space under the norm $$\|f\|_{BMOA} = |f(0)| + \|f\|_*.$$   

It is well-known that $BMOA$ is a Banach space of analytic functions, a subspace of the Bloch space, and contains $H^\infty$ as a subspace since $\|f\|_{BMOA} \leq 3\|f\|_\infty$.  Thus property (i) is satisfied by $BMOA$.  The following result shows property (ii) is satisfied by $BMOA$ also.   

\begin{theorem}
Let $\varphi$ be an analytic self-map of $\D$.  Then $C_\varphi$ acting on $BMOA$ is bounded and $\rho(C_\varphi) = 1$.
\end{theorem}

\begin{proof}
As a result of the Littlewood subordination principle (see Theorem 1.7 of \cite{Duren:1970}), every analytic self-map $\varphi$ of $\D$ induces a bounded composition operator on $BMOA$.

To compute the spectral radius of $C_{\varphi}$ acting on $BMOA$, we first estimate the norm.  By  Corollary 2.2 of \cite{Laitila:2009}, there is a constant $C>0$, independent of $\varphi$, such that 
\begin{equation}\label{eqn:normBMOA}\|C_{\varphi}\|\leq C\left(\sup_{a\in\D}\|\tau_{\varphi(a)}\circ\varphi\circ\tau_a\|_{H^2}+\log\frac2{1-|\varphi(0)|^2}\right).
\end{equation}  Since the function $\tau_{\varphi(a)}\circ\varphi\circ\tau_a$ is a composition of self-maps of the disk, the first term on the right is bounded above by 1.  Also, $$\frac{1}{1-|\varphi(0)|^2}\leq\frac{1+|\varphi(0)|}{1-|\varphi(0)|} \leq \frac2{1-|\varphi(0)|}$$ and hence $$\log\left(\frac{2}{1-|\varphi(0)|^2}\right)\leq\log\left(\frac{4}{1-|\varphi(0)|}\right)\leq 2\log2-\log\left(1-|\varphi(0)|\right).$$  Applying these estimates to Eq.~(\ref{eqn:normBMOA}), we have $$\|C_{\varphi}\|\leq 
C(1+2\log2)-C\log(1-|\varphi(0)|).$$  This immediately implies that $$\|C_{\varphi_n}\|\leq C(1+2\log2)-C\log(1-|\varphi_n(0)|)$$ and it follows that $\rho(C_{\varphi})=1$ for all bounded composition operators acting on $BMOA$ by Theorem 3.7 of \cite{Donaway:1999}.
\end{proof}

Thus $BMOA$ satisfies all the properties of $\mathcal{X}$.

\subsection{Comparisons}\label{Section:Comparison}
We now investigate spaces that do not satisfy the properties of $\mathcal{X}$.  We compare the spectrum of induced composition operators on these spaces with those on $\mathcal{X}$.

\subsubsection{Hardy spaces}\label{Section:Hardy Spaces} For $1 \leq p < \infty$, the Hardy space, denoted $H^p = H^p(\D)$, is the space of analytic functions on $\D$ such that $$\|f\|_{H^p}^p = \sup_{0<r<1}\, \int_{\D}|f(re^{i\theta})|^p\frac{d\theta}{2\pi} < \infty.$$  Under this norm, the Hardy spaces are Banach spaces and for $p=2$ it is a Hilbert space.

It is well known that $H^p$ is a Banach space of analytic functions that contains $H^\infty$ as a subspace.  For the cases of an elliptic or parabolic automorphism $\varphi$, it is the case that $\rho(C_\varphi) = 1$ and the spectrum of $C_\varphi$ on $H^p$ is the same as for $C_\varphi$ acting on $\mathcal{X}$ (see Theorem 3.9 of \cite{CowenMacCluer:1995}).  However, it is not the case that the spectral radius is 1 for every composition operator induced by an automorphism.  In fact, if $\varphi$ is hyperbolic, then $\rho(C_\varphi) = \varphi'(a)^{-1/p}$ where $a$ is the Denjoy-Wolff point of $\varphi$ (see Theorem 3.9 of \cite{CowenMacCluer:1995}).  In this situation, $\varphi'(a) < 1$ thus making $\rho(C_\varphi) > 1$.  In turn, the spectrum is the annulus $\varphi'(a)^{1/p} \leq |z| \leq \varphi'(a)^{-1/p}$ (see Theorem 4.9 of \cite{OH:2013}).

\subsubsection{Weighted Bergman spaces} For $1 \leq p < \infty$ and $\alpha > -1$, the standard weighted Bergman space, denoted $A_\alpha^p = A_\alpha^p(\D)$, is the space of analytic functions on $\D$ such that $$\|f\|_{A_\alpha^p}^p = \int_{\D} (1-|z|^2)^\alpha|f(z)|^p\;dA(z) < \infty,$$ where $dA(z)$ is the normalized Lebesgue area measure on $\D$.  The weighted Bergman spaces are Banach spaces under the norm $\|\cdot\|_{A_\alpha^p}$.

It is well known that $A_\alpha^p$ is a Banach space of analytic functions that contains $H^\infty$ as a subspace.  For the cases of an elliptic or parabolic automorphism $\varphi$, it is the case that $\rho(C_\varphi) = 1$ and the spectrum of $C_\varphi$ on $A_\alpha^p$ is the same as for $C_\varphi$ acting on $\mathcal{X}$ (see Lemma 4.2 and Theorem 4.14 of \cite{OH:2013}).  However, as was the case for the Hardy spaces, it is not the case that the spectral radius is 1 for every composition operator induced by an automorphism.  In fact, if $\varphi$ is hyperbolic, then $$\rho(C_\varphi) = \max\left\{\frac{1}{\varphi'(a)^s},\frac{1}{\varphi'(b)^s}\right\}$$ where $s = \frac{\alpha+2}{p}$, $a$ is the Denjoy-Wolff point and $b$ is the other fixed point of $\varphi$ (see Theorem 4.6 of \cite{OH:2013}).  In turn, the spectrum contains the annulus $$\min\left\{\frac{1}{\varphi'(a)^s},\frac{1}{\varphi'(b)^s}\right\} \leq |z| \leq \max\left\{\frac{1}{\varphi'(a)^s},\frac{1}{\varphi'(b)^s}\right\}$$ (see Corollary 4.7 of \cite{OH:2013}).

\subsubsection{Weighted Banach spaces} For $0 < p < \infty$, the standard weighted Banach space on $\D$, denoted $H^\infty_p = H^\infty_p(\D)$, is the space of analytic functions on $\D$ such that $$\|f\|_{H^\infty_p} = \sup_{z \in \D}\,(1-|z|^2)^p|f(z)| < \infty,$$  The weighted Banach spaces are, not surprising, Banach spaces under the norm $\|\cdot\|_{H^\infty_p}$.

It is well known that $H^\infty_p$ is a Banach space of analytic functions that contain $H^\infty$ as a subspace.  For the cases of an elliptic or parabolic automorphism $\varphi$, it is the case that $\rho(C_\varphi) = 1$ and the spectrum of $C_\varphi$ on $H^\infty_p$ is the same as for $C_\varphi$ acting on $\mathcal{X}$ (see Lemma 4.2 and Theorem 4.14 of \cite{OH:2013}).  However, as was the case for the Hardy spaces, it is not the case that the spectral radius is 1 for every composition operator induced by an automorphism.  In fact, if $\varphi$ is hyperbolic, then $$\rho(C_\varphi) = \max\left\{\frac{1}{\varphi'(a)^s},\frac{1}{\varphi'(b)^s}\right\}$$ where $a$ is the Denjoy-Wolff point and $b$ is the other fixed point of $\varphi$ (see Theorem 4.6 of \cite{OH:2013}).  In turn, the spectrum contains the annulus $$\min\left\{\frac{1}{\varphi'(a)^p},\frac{1}{\varphi'(b)^p}\right\} \leq |z| \leq \max\left\{\frac{1}{\varphi'(a)^p},\frac{1}{\varphi'(b)^p}\right\}$$ (see Corollary 4.7 of \cite{OH:2013}).

\subsubsection{Dirichlet space} The Dirichlet space on $\D$, denoted $\mathcal{D}$, is the space of analytic functions on $\D$ such that $$\int_{\D}|f'(z)|^2\;dA(z) < \infty$$ where $dA$ denotes the normalized Lebesgue area measure on $\D$.  Under the norm $$\|f\|_\mathcal{D}^2 = |f(0)|^2 + \int_{\D}|f'(z)|^2\;dA(z)$$ the Dirichlet space has a Hilbert space structure.  Although not every analytic self-map of $\D$ induce bounded composition operators on $\mathcal{D}$, univalent maps, and thus the automorphisms, of $\D$ do.

Independently, Donaway (Corollary 3.11 of \cite{Donaway:1999}) and Mart\'in and Vukoti\'c (Theorem 7 of \cite{MartinVukotic:2005}) showed that composition operators on $\mathcal{D}$ induced by univalent self-maps of $\D$, and thus the automorphisms, have spectral radius 1.  However, by direct calculation one can see that the functions in $\mathcal{F}$ are not contained in the Dirichlet space; for the case of $f_s$ this is shown in \cite{Pons:2010} (see pg. 455). Despite $\mathcal{D}$ not satisfying all the properties of $\mathcal{X}$, the spectrum of automorphism induced composition operators on $\mathcal{D}$ are precisely the same as those on $\mathcal{X}$.  

To overcome the lack of eigenfunctions, the authors in \cite{Higdon:1997} and \cite{GutierrezRodriguez:2003} used two new approaches.  In \cite{Higdon:1997}, the author produces approximate eigenfunctions and in \cite{GutierrezRodriguez:2003} unitary similarity is the key tool.

\begin{remark}
For all of the spaces discussed in Sections \ref{Section:Examples} and \ref{Section:Comparison} (and those discussed in the next section), the spectrum of $C_\varphi$ when $\varphi$ is elliptic will be the same as that for $C_\varphi$ acting on $\mathcal{X}$.  This is due to the fact that the eigenfunctions are the monomials, which are contained in all of these spaces.  
\end{remark}

\section{Open Questions}
We end this paper with open questions which were inspired while developing the examples and comparisons in Sections \ref{Section:Examples} and \ref{Section:Comparison}.  

\subsection{The little Bloch space}  While the Bloch space contains the polynomials, they are not dense in $\mathcal{B}$.  The closure of the polynomials in $\|\cdot\|_{\mathcal{B}}$ is called the little Bloch space, denoted $\mathcal{B}_0 = \mathcal{B}_0(\D)$.  More formally, the little Bloch space consists of the functions $f \in \mathcal{B}$ such that $$\lim_{|z| \to 1} (1-|z|^2)|f'(z)| = 0.$$

From Theorem 12 of \cite{ArazyFisherPeetre:1985}, bounded composition operators on $\mathcal{B}_0$ are induced exactly by functions \textit{in} $\mathcal{B}_0$, which include the automorphisms.  Donaway also proved the spectral radius of every bounded composition operator on $\mathcal{B}_0$ is 1.   Thus property (ii) is satisfied by $\mathcal{B}_0$. However, the following result shows that $\mathcal{F}$ is not contained in $\mathcal{B}_0$, and thus property (i) of $\mathcal{X}$ is not satisfied.

\begin{theorem}\label{LittleBlochSpace}
The functions $f_s$ and $f_t$, for $s > 0$ and $t \neq 0$, are not contained in the little Bloch space.
\end{theorem}

\begin{proof}
Consider the function $$f_t(z)=\exp\left(it\log\frac{1+z}{1-z}\right).$$ We show that this function is not in $\mathcal{B}_0$ for $t\in\mathbb{R}\setminus\{0\}.$  Taking the derivative, $$\begin{aligned}f'_t(z)=f_t(z)\left(it\frac{1-z}{1+z}\right)\frac{2}{(1-z)^2}&=f_t(z)\frac{2it}{(1-z)(1+z)}.\end{aligned}$$

For $t>0$, $|f_t(z)|\geq e^{-2\pi}$ and, for $t<0$, $|f_t(z)|\geq 1$. In either case, there is a constant $C>0$ such that $|f_t(z)|\geq C$ for all $t\in \R\setminus\{0\}$ and all $z\in\mathbb{D}$.  Hence $$|f'_t(z)|\geq\frac{2C|t|}{|z-1||z+1|}.$$

To show that $f_t\not\in\mathcal{B}_0$, we need to show that $$\lim_{|z|\rightarrow 1}(1-|z|^2)|f'_t(z)|\neq 0.$$  To see this, first observe that $$\lim_{|z|\rightarrow 1}(1-|z|^2)|f'_t(z)|\geq\lim_{|z|\rightarrow 1}(1-|z|^2)\frac{2C|t|}{|z-1||z+1|}$$ by our estimate from above.  If we now take a radial path to 1, that is, we set $z=r$ and let $r\uparrow 1$, we have $$\lim_{r\rightarrow 1^-}(1-r^2)\frac{2C|t|}{(1-r)(1+r)}=2C|t|>0$$ when $t\neq 0$.  Thus $$\lim_{|z|\rightarrow 1}(1-|z|^2)\frac{2Ct}{|z-1||z+1|}\neq 0$$ for $t \neq 0$, and hence $f_t$ is not in $\mathcal{B}_0$.

Next consider the function $$f_s(z)=\exp\left(\frac{s(z+1)}{z-1}\right).$$ We will show that this function is not in $\mathcal{B}_0$ for $s>0$.  First observe that $$f'_s(z)=\exp\left(\frac{s(z+1)}{z-1}\right)\left(\frac{-2s}{(z-1)^2}\right)$$ and thus we aim to show that 
$$\lim_{|z|\rightarrow 1}(1-|z|^2)|f'_s(z)|=\lim_{|z|\rightarrow 1}(1-|z|^2)\left|\exp\left(\frac{s(z+1)}{z-1}\right)\right|\frac{2s}{|1-z|^2}\neq 0.$$

Fix $x_0<0$ and consider the sequence $\{z_n\}$ defined by $$z_n=\frac{x_0+in+1}{x_0+in-1}.$$  Since $x_0<0$, this sequence is contained in the unit disk and $\{z_n\}\rightarrow1$ as $n\rightarrow \infty.$ To obtain our conclusion, we show $$\lim_{n\rightarrow\infty}(1-|z_n|^2)\left|\exp\left(\frac{s(z_n+1)}{z_n-1}\right)\right|\frac{2s}{|1-z_n|^2}\neq 0.$$

First observe that the map $\psi(z)=(z+1)/(z-1)$ is its own inverse and hence $\psi(z_n)=x_0+in$ for each $n\in \mathbb{N}$.  Thus $$\left|\exp\left(\frac{s(z_n+1)}{z_n-1}\right)\right|=|\exp(sx_0+isn)|=e^{sx_0}>0.$$  Substituting, $$\lim_{n\rightarrow\infty}(1-|z_n|^2)\left|\exp\left(\frac{s(z_n+1)}{z_n-1}\right)\right|\frac{2s}{|1-z_n|^2}
=\lim_{n\rightarrow\infty}e^{sx_0}(1-|z_n|^2)\frac{2s}{|1-z_n|^2}.$$ Next, $$1-|z_n|^2=\frac{-4x_0}{(x_0-1)^2+n^2}$$ and $$|1-z_n|^2=\frac{4}{(x_0-1)^2+n^2}.$$  Thus $$\begin{aligned}\lim_{n\rightarrow\infty}e^{sx_0}(1-|z_n|^2)\frac{2s}{|1-z_n|^2}&=
\lim_{n\rightarrow\infty}e^{sx_0}\left(\frac{-4x_0}{(x_0-1)^2+n^2}\right)\left(\frac{s((x_0-1)^2+n^2)}{2}\right)\\
&=\lim_{n\rightarrow\infty}(-2sx_0)e^{sx_0}>0\end{aligned}$$ and hence $f_s$ is not in $\mathcal{B}_0$ for $s > 0$.
\end{proof}

\noindent For the little Bloch space, we leave the reader with the following question.

\begin{question} For $\varphi$ a parabolic or hyperbolic automorphism, what is the spectrum of $C_\varphi$ on the little Bloch space?
\end{question}

\subsection{Analytic functions of vanishing mean oscillation} Like the Bloch space, $BMOA$ contains the polynomials, but they are not dense in $BMOA$.  The closure of the polynomials in $\|\cdot\|_{\mathrm{BMOA}}$ is denoted by $VMOA$.  $VMOA$ is the space of analytic functions with vanishing mean osciallation on $\partial\D$, formally defined as the functions $f \in \mathrm{BMOA}$ such that
$$\lim_{|a|\to 1}\,\|f\circ\tau_a - f(a)\|_{H^2} = 0.$$  

By Corollary 4.2 of \cite{Laitila:2009}, $C_\varphi$ is bounded on $VMOA$ if and only if $\varphi \in VMOA$.  So every automorphism induces a bounded composition operator on $VMOA$.  By the same argument as in Section \ref{Section:BMOA}, the spectral radius of $C_\varphi$ induced by a disk automorphism is 1.  Thus property (ii) of $\mathcal{X}$ is satisfied.  Since $VMOA$ is a subspace of the little Bloch space (see \cite{GG:2013}), it follows that $VMOA$ does not satisfy property (i), a corollary of Theorem \ref{LittleBlochSpace}.

\begin{corollary}
The functions $f_s$ and $f_t$, for $s > 0$ and $t \neq 0$, are not contained in $VMOA$.
\end{corollary} 

\noindent For $VMOA$, we leave the reader with the following question.

\begin{question} For $\varphi$ parabolic or hyperbolic automorphism, what is the spectrum of $C_\varphi$ on $VMOA$?
\end{question} 

\section*{Acknowledgements}
\noindent The work of the second author was conducted while an undergraduate student at the University of Wisconsin-La Crosse and funded by the College of Science and Health Dean's Distinguished Fellowship.  

\bibliographystyle{amsplain}
\bibliography{AllenLePons.bib}

\end{document}